\documentclass[12pt]{article}

\usepackage{amsmath,amssymb,amsthm}
\usepackage{authblk}
\usepackage{mathrsfs}

\usepackage{fullpage}
\usepackage{color}
\usepackage{enumerate}

\newtheorem{theorem}{Theorem}
\newtheorem{proposition}[theorem]{Proposition}

\title{Equivalence of Pseudo- and Approximate Amenability in Banach Algebras}

\author[1]{Fereidoun Ghahramani%
  \thanks{Email: \texttt{Fereidoun.Ghahramani@umanitoba.ca}. Supported by NSERC Grant 2021--05476.}}

\author[1]{Maedeh Soroushmehr%
  \thanks{Email: \texttt{m.soroushmehr@gmail.com}.}}

\author[1]{Yong Zhang%
  \thanks{Email: \texttt{Yong.Zhang@umanitoba.ca}. Supported by NSERC Grant 2022--04137.}}

\affil[1]{University of Manitoba,
Winnipeg, MB R3T 2N2, Canada}

\date{}

\begin{document}

\maketitle

\begin{abstract}
We prove that pseudo-amenability implies approximate amenability for Banach algebras with a bounded approximate identity. Consequently, these two notions of generalized amenability are equivalent for such Banach algebras. Our proof closes a gap in the paper by F. Ghahramani and Y. Zhang, Pseudo-amenable and pseudo-contractible Banach algebras, Math. Proc. Cambridge Philos. Soc. 142 (2007), 111–123.
\end{abstract}

\bigskip

\medskip


\bigskip

The notions of approximate amenability and pseudo-amenability for Banach algebras were introduced in \cite{gl} and \cite{gz}, respectively (see also \cite{cgz,glz,gs} for further developments). The relationship between these two notions was investigated in \cite{gz}, where it was shown that approximate amenability implies pseudo-amenability for Banach algebras with a bounded approximate identity. The converse was also claimed there. However, the proof of this implication contains a gap (see \cite{z}).

The purpose of this note is to establish this implication, thereby completing the proof that approximate amenability and pseudo-amenability are equivalent for Banach algebras with a bounded approximate identity.

We refer the reader to \cite{gl,gz} for the definitions of these notions. Our argument, however, does not depend on them.

If two Banach algebras $A$ and $B$ are isometrically isomorphic, we will write $A\cong B$.

Let $(A_i)_{i \in \Lambda}$ be a family of Banach algebras.
Following the notation of \cite{gz}, the $c_0$-direct sum of the family is
\[
\bigoplus^0_{i \in \Lambda} A_i = \left\{ (a_i)_{i \in \Lambda} \in \prod_{i \in \Lambda} A_i : \lim_{i} \|a_i\|_{A_i} = 0\right\}. 
\]
With coordinatewise algebraic operations and the norm 
\[
\| (a_i)_{i \in \Lambda} \| = \sup_{i \in \Lambda} \|a_i\|_{A_i} ,
\] 
$\bigoplus^0_{i \in \Lambda} A_i$ is a Banach algebra.

Denote by $\mathscr{B}$ the class of all Banach algebras. The isometric isomorphism relation $\cong$ is an equivalence relation on $\mathscr{B}$. Moreover, if $A_i\cong B_i$ for all $i\in\Lambda$, then
\[
\bigoplus^0_{i\in\Lambda} A_i \cong \bigoplus^0_{i\in\Lambda} B_i.
\]
Let
\[
\widetilde{\mathscr{B}}=\mathscr{B}/\!\cong
\]
denote the quotient of $\mathscr{B}$ by this equivalence relation.  
For $A\in\mathscr{B}$, write $[A]$ for its equivalence class:
\[
[A]=\{B\in\mathscr{B}: B\cong A\}.
\]
We may therefore define the $c_0$-direct sum on $\widetilde{\mathscr{B}}$ by
\[
\bigoplus^0_{i\in\Lambda}[A_i]
:=
\left[\bigoplus^0_{i\in\Lambda}A_i\right],
\]
where $(A_i)_{i\in\Lambda}\subset\mathscr{B}$.

\begin{proposition}\label{zorn}
Let $A$ be a pseudo-amenable Banach algebra with a bounded approximate identity. Then $A$ is approximately amenable.
\end{proposition}

\begin{proof}
Let $d\geq 1$ be fixed.

Denote by $\Sigma_0$ the set of all $[A]\in\widetilde{\mathscr{B}}$ such that $A$ is pseudo-amenable and has a bounded approximate identity of bound at most $d$, and let $\Sigma$ be the subset of $\Sigma_0$ consisting of those $[A]$ for which $A$ is not approximately amenable.

Since pseudo-amenability, approximate amenability, and the bound of a bounded approximate identity are all preserved under isometric isomorphisms, the sets $\Sigma_0$ and $\Sigma$ are well defined. Moreover, the $c_0$-direct sum preserves the common upper bound on the bounded approximate identities of the summand algebras. 
Hence, by \cite[Proposition~2.1]{gz} and \cite[Corollary~2.3]{gl}, both $\Sigma_0$ and $\Sigma$ are closed under $c_0$-direct sums. 
Furthermore, if $[A]\in\Sigma$ and $[J]\in\Sigma_0$, then
\[
[A]\oplus^0[J]=[A\oplus^0J]\in\Sigma.
\]

We shall prove that $\Sigma=\varnothing$.

Suppose, towards a contradiction, that $\Sigma \neq \varnothing$.  
We define a relation $\preceq$ on $\Sigma$ as follows: for  $[A], [B] \in \Sigma$,  we write $[A] \preceq [B]$ if and only if the following conditions hold: 

\begin{enumerate}[(a)]
\item There exists $[J]\in \Sigma_0$ such that 
\[
[B] =[A]{\oplus}^0 [J]. 
\]
\label{A<B}
\item Whenever there  exists $[J']\in \Sigma_0$ such that 
\[
[A] = [B] \oplus^0 [J'],
\]
we have $[A] = [B]$. 
\label{A=B}
\end{enumerate} 

We now verify that $(\Sigma, \preceq)$ is a partially ordered set.

\begin{enumerate} 
\item The trivial algebra $\bf 0 =\{0\}$ belongs to $\Sigma_0$, and 
\[
[A] = [A] \oplus^0 \bf 0.
\]  
This ensures the reflexivity $[A]\preceq[A]$ holds for $[A]\in \Sigma$.

\item To show antisymmetry, let $[A], [B] \in \Sigma$, and suppose that $[A]\preceq [B]$ and $[B]\preceq [A]$. 
By condition \eqref{A<B},  there exist $[J], [J']\in \Sigma_0$ such that 
\[
[B] = [A] \oplus^0 [J], \quad \text{and } [A] = [B] \oplus^0 [J'].
\]
Applying the condition \eqref{A=B} ensured for $[A]\preceq [B]$,  we obtain $[A] = [B]$. Thus, antisymmetry holds.

\item For transitivity, suppose $[A] \preceq [B]$ and $[B] \preceq [C]$.

By condition \eqref{A<B}, there are $[J_1],[J_2]\in \Sigma_0$ such that
\[
[B]=[A] \oplus^0 [J_1] \quad \text{and} \quad [C]=[B] \oplus^0 [J_2].
\]
Then
\[
[C] =([A] \oplus^0 [J_1]) \oplus^0 [J_2] =[A] \oplus^0 ([J_1] \oplus^0 [J_2]).
\]
Thus, condition \eqref{A<B} holds for the pair $[A]$ and $[C]$.

To verify condition \eqref{A=B} for $[A]$ and $[C]$, suppose that there is $[J']\in \Sigma_0$ such that $[A] = [C] \oplus^0 [J']$. Since $[B]=[A] \oplus^0 [J_1]$, it follows that
\[
[B]=[C] \oplus^0 ([J']\oplus^0[J_1])=[C] \oplus^0 [J'\oplus^0J_1].
\]
 Hence, condition \eqref{A=B} for $[B] \preceq [C]$ implies that $[B] = [C]$, whence $[A] \preceq [C]$ since $[A] \preceq [B]$. Thus, $\preceq$ is transitive.
\end{enumerate}

Therefore, $(\Sigma, \preceq)$ is a partially ordered set.

Let $([A_\lambda])_{\lambda\in\Lambda}$ be a chain in $\Sigma$.

To show that the chain has an upper bound in $\Sigma$, without loss of generality, we may assume that the chain is strictly increasing; that is,
\[
[A_\lambda]\preceq [A_{\lambda'}]
\quad\text{and}\quad
[A_\lambda]\neq [A_{\lambda'}]
\]
whenever $\lambda<\lambda'$.

Then
\[
\bigoplus_{\lambda\in\Lambda}^0 [A_\lambda]\in\Sigma,
\]
since $\Sigma$ is closed under $c_0$-direct sums. Let
$M=\bigoplus_{\lambda\in\Lambda}^0 A_\lambda$.
We indeed have
\[
\bigoplus_{\lambda\in\Lambda}^0 [A_\lambda]=[M].
\]
We now show $[A_\lambda]\preceq [M]$ for each $\lambda\in \Lambda$. 
We may write
$[M] = [A_\lambda]\oplus^0 [J_\lambda]$, where
\[
[J_\lambda] = \left[
\bigoplus_{\substack{\mu\in\Lambda\ \mu\neq\lambda}}^0 A_\mu
\right]
\in\Sigma_0.
\]
So condition (a) holds for the pair $[A_\lambda]$ and $[M]$. 

To verify \eqref{A=B}, it suffices to show that there is no $[J']\in\Sigma_0$ such that
\[
[A_\lambda]=[M]\oplus^0[J'].
\]

Suppose, to the contrary, that such a $[J']$ exists. Then, for each $\lambda'>\lambda$, we have
\[
[A_\lambda]
=[M]\oplus^0[J']
=[A_{\lambda'}]\oplus^0\bigl([J_{\lambda'}]\oplus^0[J']\bigr).
\]
Since $\Sigma_0$ is closed under $c_0$-direct sums, we have
\[
[J_{\lambda'}]\oplus^0[J']\in\Sigma_0.
\]
Applying \eqref{A=B} for $[A_\lambda]\preceq[A_{\lambda'}]$, we conclude that
$[A_\lambda]\cong[A_{\lambda'}]$
for every $\lambda'>\lambda$, contradicting the assumption that the chain is strictly increasing.

Therefore, \eqref{A=B} holds. Hence $[A_\lambda]\preceq[M]$ for every $\lambda\in\Lambda$, so $[M]$ is an upper bound of the chain in $\Sigma$.

By Zorn's lemma, $(\Sigma,\preceq)$ has a maximal element, say $[A]$.

Suppose that $\operatorname{card}(A)=\kappa$, which is clearly invariant under isometric isomorphisms. By the Axiom of Choice, there exists a set $I$ with $\operatorname{card}(I)=2^\kappa$.

By \cite[Proposition 2.1]{gz},
\[
J:=c_0(I)=\bigoplus_{i\in I}^0\mathbb{C}
\]
is pseudo-amenable and has a bounded approximate identity of bound $1\le d$. Hence $[J]\in\Sigma_0$, and therefore
\[
[B]:=[A]\oplus^0[J]\in\Sigma.
\]
On the other hand, there is no $[J']\in \Sigma_0$ such that 
$[A]=[B]\oplus^0[J']$.
Otherwise, 
\[
A\cong A\oplus^0J\oplus^0J',
\]
and hence
\[
\kappa=\operatorname{card}(A)
\geq \operatorname{card}(J)
=2^\kappa,
\]
which is impossible. Thus, \eqref{A=B} holds for $[A]$ and $[B]$, and we have $[A]\preceq[B]$.

Since $[A]$ is maximal, we must have $[A]=[B]$, that is, $A\cong B$.  Then we will have
\[
\operatorname{card}(A)=\operatorname{card}(B)\geq \operatorname{card}(J)=2^\kappa>\kappa=\operatorname{card}(A),
\]
a contradiction again.
The contradiction shows that we must have $\Sigma=\varnothing$. 
Consequently, every pseudo-amenable Banach algebra with a bounded approximate identity is approximately amenable. 
This completes the proof.
\end{proof}

Proposition~\ref{zorn} fills a gap in the proof of \cite[Proposition~3.2]{gz}. We conclude this note by restating that proposition as a theorem.

\begin{theorem}
For a Banach algebra $A$, the following are equivalent.
\begin{enumerate}[(i)]
\item $A$ is pseudo-amenable and has a bounded approximate identity.
\item $A$ is approximately amenable and has a bounded approximate identity.
\item $A$ has an approximate diagonal $(m_\mu)_{\mu\in \Gamma}\subset A\hat{\otimes}A$ such that $(\pi(m_\mu))_{\mu\in \Gamma}$ is bounded.
\end{enumerate}
\end{theorem}

\end{document}